\documentclass[12pt]{amsart}
\usepackage{amssymb,latexsym}
\usepackage{enumerate}
\usepackage{hyperref}

\makeatletter \@namedef{subjclassname@2010}{%
  \textup{2010} Mathematics Subject Classification}
\makeatother

\newtheorem{Theorem}{Theorem}

\newtheorem{Lemma}{Lemma}
\newtheorem{Corollary}{Corollary}

\newtheorem*{Conjecture}{Conjecture}
\newtheorem*{UnnumberedTheorem}{Theorem}

	\newcommand{\ZZ}[0]{\mathbb Z}

	\newcommand{\cP}[0]{\mathcal P}

\newcommand{\cL}[0]{\mathcal L}	
	
\newcommand{\eps}[0]{\varepsilon}

\newcommand{\lr}[1]{\left(#1\right)}

\begin{document}


\baselineskip=17pt



\title{The Additive Structure of Cartesian Products Spanning Few Distinct Distances}

\author[Brandon Hanson]{Brandon Hanson} \address{Pennsylvania State University\\
University Park, PA}
\email{bwh5339@psu.edu}

\begin{abstract}
Guth and Katz proved that any point set $\cP$ in the plane determines $\Omega(|\cP|/\log|\cP|)$ distinct distances. We show that when near to this lower bound, a point set $\cP$ of the form $A\times A$ must satisfy $|A-A|\ll |A|^{2-1/8}$.
\end{abstract}

\maketitle

\section{Introduction}

If $\cP$ is a set of $N$ points in the plane, let $\Delta(\cP)$ denote the set of (squared) distances spanned by $\cP$, that is
\[\Delta(\cP)=\{(p_1-q_1)^2+(p_2-q_2)^2:(p_1,p_2),(q_1,q_2)\in\cP\}.\] Erd\H{o}s famously conjectured \cite{E1} that any set of points needs to determine at least a $\Omega(N/\sqrt{\log N})$ distinct numbers as distances. This lower bound occurs when the points are arranged in a square grid in the integer lattice, for instance. In a landmark paper \cite{GK}, Guth and Katz very nearly established Erd\H{o}s' conjecture, by proving the lower bound:

\begin{Theorem}[Guth-Katz]\label{GuthKatz}
If $\cP$ is a set of $N$ points in the plane then $|\Delta(\cP)|\gg \frac{N}{\log N}$.
\end{Theorem}

Though the problem of estimating the number of distinct distances is now almost resolved, what remains to be shown is a characterization of those point sets for which the lower bound of Theorem \ref{GuthKatz} is sharp. In this note, we think of an extremal configuration as a set of points $\cP$ spanning $O(|\cP|^{1+\eps})$ distances. The known examples of such sets all appear to come from sets with algebraic structure - in the grid example above, the points are coming from a lattice. Another example is to take the vertices of a regular $n$-gon, or $n$ equally spaced points on a line. One might ask if this algebraic structure is a necessary feature of such point sets - in the listed examples, a very strong rotational or translational symmetry is present. Such a conjecture was made by Erd\H{o}s in \cite{E2}, where he states that all extremal configurations should exhibit a lattice-like structure. Some progress towards such a result is given in \cite{SZZ}. Erd\H{o}s is pretty vague about what he means by lattice like, but we are going to take a lattice to mean a discrete subgroup of the plane. Here, we will prove a theorem in further support of this fact. Let $A$ be a set of real numbers such that $\cP=A\times A$ determines few distances. Theorem \ref{GuthKatz} says that there are at least $c|A|^2/\log |A|$ distinct distances for some positive, absolute constant $c$. We will show that when near this lower bound, the difference set \[D=A-A=\{a_1-a_2:a_1,a_2\in A\}\] has to be somewhat small, thus showing that the set $\cP$ is to some extent additively structured. It can certainly be argued that the assumption that $\cP$ is a cartesian product means that our point set is already lattice-like, but I think that the additive behaviour of $A$ is essential to be considered truly lattice-like. 

\begin{UnnumberedTheorem}[Main theorem]
Suppose $A$ is a finite set of real numbers and let $\Delta(A\times A)$ be the set of distances spanned by $A\times A$. Then
\[|A-A|\ll |\Delta(A\times A)||A|^{-1/8}.\] 
\end{UnnumberedTheorem}

In \cite{E2}, it was conjectured that in extremal point configurations there should be many points ($|\cP|^{1/2}$ is conjectured, but $|\cP|^\eps$ is already interesting) on a line or circle. In the case of a cartesian product, there are trivially many points on axis parallel lines, or the main diagonal. However, this theorem also shows that there are non-trivial lines which contain many points.

\begin{UnnumberedTheorem}[Rich lines]
Suppose $A$ is a finite set of real numbers such that $\Delta(A\times A)=O(|A|^2)$. There is a non-trivial line (in fact many) which contains $\Omega(|A|^{1/8})$ points of $A\times A$.
\end{UnnumberedTheorem}
\begin{proof}
Since \[|A|^2=\sum_{d\in A-A}\sum_{\substack{a,b\in A\\a-b=d}}1,\] one of the inner sums is at least $\Omega(|A|^{1/8})$. Thus there are $\Omega(|A|^{1/8})$ points $(a,b)\in A\times A$ with $a-b=d$, and this equation defines the desired line.
\end{proof}

From an arithmetic combinatorial point of view, our main theorem ought to be true because the set of distances spanned by $\cP$ is the set of numbers $\Delta=D^2+D^2$, where $D^2=\{d^2:d\in D\}$. Heuristically, the squaring of an additively structured set (in this case $D$, a difference set) should result in a set (in this case $D^2$) which is not additively structured. The fact that squaring, or in fact any convex function, meddles with the additive properties of a set of numbers is captured in a theorem of Elekes-Nathanson-Ruzsa from \cite{ENR}, later improved by Li and Roche-Newton \cite{LRN}.

\begin{Theorem}[Convexity and Sumsets]\label{ENR}
Let $S$ be a finite set of real numbers, let $f$ be a strictly convex function and let $\eps>0$. Then
\[\max\{|S+S|,|f(S)+f(S)|\}\gg|S|^{14/11-\eps}.\] In particular,
\[\max\{|S+S|,|S^2+S^2|\}\gg|S|^{14/11-\eps}.\]
\end{Theorem}

Viewed in this way, the main theorem is really one in arithmetic combinatorics. Recently, Shkredov has proved similar theorems, examining the multiplicative structure (or lack thereof) of difference sets, see \cite{S}.

We close this introduction by remarking that if a well-known conjecture of Rudin holds (this form is due to Ruzsa, \cite{CG}), then a very strong improvement of the main theorem can be made to points with integral co-ordinates.

\begin{Conjecture}[Rudin/Ruzsa]
For $\eps>0$, if $S$ is a set of perfect squares then $|S+S|\gg_\eps |S|^{2-\eps}$. 
\end{Conjecture}

Since $D=A-A\subset \ZZ$, $D^2$ is a set of perfect squares. It follows that $\Delta(A\times A)\gg_\eps |D|^{2-\eps}$. If $\Delta(A\times A)=O(|A|^2)$ then $|D|=O(|A|^{1+\delta})$ for some $\delta\to 0$ as $\eps\to 0$.

\section{Proofs}

The fundamental observation in this paper is a simple one. By taking particular elements of \[2D^2-2D^2=\{d_1^2+d_2^2-d_3^2-d_4^2:d_1,d_2,d_3,d_4\in D\}\] we can find a dilated copy of $D\cdot D$.

\begin{Lemma}\label{Differencing}
If $D$ is a difference set then $2D^2-2D^2$ contains a dilate of the set $D\cdot D$.
\end{Lemma}
\begin{proof}
Let $D=A-A$ and suppose $d_1=a_2-a_1$ and $d_2=b_1-b_2\in D$. Then
\[(b_1-a_1)^2+(b_2-a_2)^2-(b_1-a_2)^2-(b_2-a_1)^2=2d_1d_2.\] The left hand side clearly shows this element of $2\cdot D\cdot D$ belongs to $2D^2-2D^2$.
\end{proof}

The other essential ingredient is the Plunnecke-Ruzsa Theorem, which is nicely proved in \cite{P}.
\begin{Theorem}[Plunnecke-Ruzsa]\label{PlunneckeRuzsa}
Suppose $S$ is a finite subset of an abelian group. Then
\[|mS-nS|\leq\lr{\frac{|S+S|}{|S|}}^{n+m}|S|.\]
\end{Theorem}

Combining these facts gives a number of options, and I am certain that there is a more efficient idea that could lead to a quantitative improvement of the main theorem. In particular, avoiding such a large exponent in the use of Theorem \ref{PlunneckeRuzsa} could drastically improve the final result. In a first iteration of this paper, we would have used that $2D^2-2D^2$ contains a dilate of $D$. Then because of Theorem \ref{ENR}, $4D^2-4D^2$ (which contains a dilate of $D+D$) is bounded from below in terms of $|D|$. We can get a better exponent by instead using an idea of Solymosi from his well-known work on the Sum-Product problem, \cite{So}.

\begin{Lemma}\label{Solymosi}
Let $\cL$ be a collection of $L$ lines in the plane which pass through the origin. Let $\cP$ be a collection of points lying in a single quadrant in the plane, and such that each line $l\in \cL$ contains at least $n$ points from $\cP$. Then
\[|\cP+\cP|\geq (L-1)n^2.\]
\end{Lemma}
\begin{proof}
Sort the lines in $\cL$ by increasing slope. Let $l_1$ and $l_2$ be adjacent lines. Then the (vector) sum of any point on $l_1$ and any point on $l_2$ lies between $l_1$ and $l_2$. This means that all sums of points in $\cP$ coming from $l_1$ and $l_2$ are distinct from the sums coming from any other two adjacent lines in $\cL$. Moreover, all such sums are distinct by linear independence, so any two adjacent lines produce $n^2$ sums. Since there are $L-1$ adjacent pairs, the lemma follows. 
\end{proof}

\begin{Corollary}\label{ProductSumset}
Let $S$ be a set of real numbers and let $S/S=\{s_1/s_2:s_1,s_2\in S\}$. Then \[|S\cdot S+S\cdot S|^2\gg |S/S||S|^2.\]
\end{Corollary}
\begin{proof}
We will apply Lemma \ref{Solymosi} to the set \[\cP=\{(s_1s,s_2s):s_1,s_2,s\in S\},\] which is a subset of $(S\cdot S)\times (S\cdot S)$. Let $r=s_1/s_2\in S/S$. Choosing $s\in S$ arbitrarily, we have $|S|$ points $(s_2s,s_1s)$, and each lies on the line $l_r$ through the origin with slope $r$. Let $\cL_1$ be the set of such lines which pass through the first and third quadrants, and $\cL_2$ the set of those that pass through the second and fourth quadrants. Then one of $|\cL_1|$ or $|\cL_2|$ has size at least $\frac{1}{2}|S/S|$. Assume $|\cL_1|\geq \frac{1}{2}|S/S|$ as the other case is similar. Next, at let $\cL_1^+$ denote those lines which have at least half of their points from $(S\cdot S)\times (S\cdot S)$ in the first quadrant and $\cL_1^-$ those with at least half of their points in the third quadrant. One of the two sets has at least $\frac{1}{4}|S/S|$ lines in it, and each line has at least $|S|/2$ points in some fixed quadrant. Now the corollary follows from Lemma \ref{Solymosi}, the fact that
\[\cP+\cP\subset (S\cdot S)\times (S\cdot S)+(S\cdot S)\times (S\cdot S),\] and the fact that \[(S\cdot S\times S\cdot S)+(S\cdot S\times S\cdot S)=(S\cdot S+S\cdot S)\times  (S\cdot S+S\cdot S).\]
\end{proof}

Finally, we recall a beautiful theorem due to Ungar (\cite{U}), which gives a lower bound for the number of slopes defined by a set of points.
\begin{Theorem}\label{Slopes}
Let $\cP$ be a finite set of points, not all on a line. Then the set \[\left\{\frac{y_1-y_2}{x_1-x_2}:(x_1,y_1),(x_2,y_2)\in \cP\right\}\] has size at least $|\cP|-1$.
\end{Theorem}

\begin{proof}[Proof of Main theorem]
Set $D=A-A$ and $\Delta=\Delta(\cP)=D^2+D^2$. From Lemma \ref{Differencing} and Theorem \ref{PlunneckeRuzsa}, we see that
\[|D\cdot D+D\cdot D|\leq|4D^2-4D^2|\leq\lr{\frac{|D^2+D^2|}{|D^2|}}^8|D^2|=\frac{|\Delta|^8}{|D|^7}.\] From Corollary \ref{ProductSumset}, we have that
\[|D\cdot D+D\cdot D|\gg |D||D/D|^{1/2}.\] But $D/D$ is the set of slopes from the set $A\times A$, so Theorem \ref{Slopes} gives \[|D\cdot D+D\cdot D|\gg |D||A|.\] Combining these estimates gives that
\[|D|\ll |\Delta||A|^{-1/8}.\] The theorem follows.
\end{proof}

\section{Acknowledgments}
I thank Oliver Roche-Newton and Adam Sheffer for much helpful discussion. I became interested in this problem while attending the IPAM reunion conference for the program \emph{Algebraic Techniques for Combinatorial and Computational Geometry}. 

\end{document}